\documentclass{article}

\usepackage{amsmath,amssymb,amsthm,bm}
\usepackage{graphicx}
\usepackage{hyperref}
\usepackage{algorithm,algorithmic}

\newcommand{\eqdef}{\stackrel{\textup{def}}{=}}

\DeclareMathOperator*{\Arg}{Arg}

\newcommand{\bigo}{\textup{O}}

\renewcommand{\phi}{\varphi}
\newcommand{\eps}{\varepsilon}

\newcommand{\Z}{\mathbb{Z}}
\newcommand{\N}{\mathbb{N}}
\newcommand{\C}{\mathbb{C}}
\newcommand{\E}{\mathbb{E}}
\renewcommand{\P}{\mathbb{P}}

\newcommand{\e}{\textup{e}}
\renewcommand{\i}{\textup{i}}
\renewcommand{\d}{\textup{d}}
\newcommand{\wt}[1]{\widetilde{#1}}
\newcommand{\wh}[1]{\widehat{#1}}

\newcommand{\poly}{\textup{poly}}

\newtheorem{thm}{Theorem}[section]

\newtheorem{lemma}[thm]{Lemma}
\theoremstyle{remark}

\usepackage[caption=false,font=footnotesize]{subfig}
\usepackage{algorithm,algorithmic,caption}
\newcommand{\wmod}[1]{\mbox{~(\textrm{mod}~$#1$)}}
\newcommand{\bS}{\bm{S}}

\title{Adaptive Sub-Linear Time Fourier Algorithms}
\author{David Lawlor, Yang Wang, and Andrew Christlieb\\Department of Mathematics, Michigan State University}
\begin{document}
%
%
%
%
%
%
\maketitle
\begin{abstract}
We present a new deterministic algorithm for the sparse Fourier transform problem, in which we seek to identify $k \ll N$ significant Fourier coefficients from a signal of bandwidth $N$. Previous deterministic algorithms exhibit quadratic runtime scaling, while our algorithm scales linearly with $k$ in the average case. Underlying our algorithm are a few simple observations relating the Fourier coefficients of time-shifted samples to unshifted samples of the input function. This allows us to detect when aliasing between two or more frequencies has occurred, as well as to determine the value of unaliased frequencies. We show that empirically our algorithm is orders of magnitude faster than competing algorithms.
\end{abstract}

%

\section{Introduction}
\label{sec:intro}
The Fast Fourier Transform (FFT) is arguably the most ubiquitous numerical algorithm in scientific computing.  In addition to being named one of the ``Top Ten Algorithms'' of the past century~\cite{dongarra2000guest}, the FFT is a critical tool in myriad applications, ranging from signal processing to computational PDE and machine learning.  At the time of its introduction, it represented a major leap forward in the size of problems that could be solved on available hardware, as it reduces the runtime complexity of computing the Discrete Fourier Transform (DFT) of a length-$N$ array from $\bigo(N^2)$ to $\bigo(N\log N)$.

Any algorithm which computes all $N$ Fourier coefficients has a runtime complexity of $\Omega(N)$, since it takes that much time merely to report the output.  However, in many applications it is known that the DFT of the signal of interest is highly sparse -- that is, only a small number of coefficients are non-zero.  In this case it is possible to break the $\Omega(N)$ barrier by asking only for the largest $k$ terms in the signal's DFT.  When $k\ll N$ existing algorithms can significantly outperform even highly optimized FFT implementations~\cite{iwen2007empirical,iwen2010combinatorial,hassanieh2012simple}.

\subsection{Related Work}
The first works to implicitly address the sparse approximate DFT problem appeared in the theoretical computer science literature in the early 1990s. In~\cite{linial1993constant}, a variant of the Fourier transform for Boolean functions was shown to have applications for learnability. A polynomial-time algorithm to find large coefficients in this basis was given in~\cite{kushilevitz1993learning}, while the interpolation of sparse polynomials over finite fields was considered in~\cite{mansour1995randomized}. It was later realized~\cite{gilbert2005improved} that this last algorithm could be considered as an approximate DFT for the special case when $N$ is a power of two.

In the past ten or so years, a number of algorithms have appeared which directly address the problem of computing sparse approximate Fourier transforms.  When comparing the results in the literature, care must be taken to identify the class of signals over which a specific algorithm is to perform, as well as to identify the error bounds of a given method.  Different algorithms have been devised in different research communities, and so have varying assumptions on the underlying signals as well as different levels of acceptable error.

The first result with sub-linear runtime and sampling requirements appeared in~\cite{gilbert2002near}. They give a $\poly(k,\log N, \log(1/\delta), 1/\eps)$ time algorithm for finding, with probability $1-\delta$, an approximation $\hat{y}$ of the DFT of the input $\hat{x}$ that is nearly optimal, in the sense that $\|\hat{x}-\hat{y}\|_2^2 \le (1+\eps)\|\hat{x}-\hat{x}_\textup{opt}\|_2^2$, where $\hat{x}_\textup{opt}$ is the best $k$-term approximation to $\hat{x}$. Here the exponent of $k$ in the runtime is two, so the algorithm is \emph{quadratic} in the sparsity.  Moreover, the algorithm is non-adaptive in the sense that the samples used are independent of the input $x$. This algorithm was modified in~\cite{gilbert2005improved} to bring the dependence on $k$ down to linear.\footnote{See \cite{gilbert2008tutorial} for a ``user-friendly'' description of the improved algorithm.} This was accomplished mainly by replacing uniform random variables (used to sample the input) by random arithmetic progressions, which allowed the use of nonequispaced fast Fourier transforms to sample from intermediate representations and to estimate the coefficients in near-linear time. The increased overhead of this procedure, however, limited the range of $k$ for which the algorithm outperformed a standard FFT implementation~\cite{iwen2007empirical}.

Around the same time, a similar algorithm was developed in the context of list decoding for proving hard-core predicates for one-way functions~\cite{akavia2003proving}. This can be considered an extension of~\cite{kushilevitz1993learning}, and like~\cite{gilbert2002near,gilbert2005improved} is a randomized algorithm.  Since the goal in this work was to give a polynomial-time algorithm for list decoding, no effort was made to optimize the dependence on $k$; it stands at $k^{11/2}$, considerably higher than~\cite{gilbert2002near,gilbert2005improved}. The randomness in this algorithm is used only to construct a sample set on which norms are estimated, and in~\cite{akavia2010deterministic} this set is replaced with a deterministic construction. This construction is based on the notion of $\eps$-approximating the uniform distribution over arithmetic progressions, and relies on existing constructions of $\eps$-biased sets of small size~\cite{katz1989estimate,ajtai1990construction}. Depending on the size of the $\eps$-biased sets used, the sampling and runtime complexities are $\bigo(k^4\log^c N)$ and $\bigo(k^6\log^c N)$, respectively, for some $c>4$.\footnote{Specifically, the runtime is $\bigo(k^2 \cdot \log N \cdot |S|)$, where $S$ is the set of samples read by the algorithm. This set takes the form $S = \bigcup_{\ell=1}^{\lfloor \log N \rfloor} A - B_\ell$, where $A$ has $\eps$-discrepancy on rank 2 Bohr sets, $B_\ell$ $\eps$-approximates the uniform distribution on $[0,2^\ell-1]\cap\Z$, and $A-B_\ell$ is the difference set. Using constructions from~\cite{katz1989estimate} one has $|A|=\bigo(\eps^{-1}\log^4 N), \; |B_\ell| = \bigo(\eps^{-3}\log^4 N)$; setting $\eps = \Theta(k^{-1})$ and noting that $\left| \bigcup A-B_\ell \right| = \bigo\left(\sum\left| A-B_\ell \right|\right)$ and $\left| A-B_\ell \right| = \bigo(|A||B_\ell|)$ (see, e.g., \cite{tao2006additive}) one obtains the stated sampling and runtime complexities.}

In the series of works \cite{iwen2008deterministic,iwen2010combinatorial,iwen2011improved}, a different deterministic algorithm for sparse Fourier approximation was given that relies on the combinatorial properties of \emph{aliasing}, or collisions among frequencies in sub-sampled DFTs.  By taking enough short DFTs of co-prime lengths, and employing the Chinese Remainder Theorem to reconstruct energetic frequencies from their residues modulo these sample lengths, the author is able to prove sampling and runtime bounds of $\bigo(k^2 \log^4 N)$.  The error bound is of the form $\|\hat{x}-\hat{y}\|_2 \le \|\hat{x}-\hat{x}_\textup{opt}\|_2 + k^{-1/2}\|\hat{x}-\hat{x}_\textup{opt}\|_1$; it has been shown that the stronger ``$\ell_2$-$\ell_2$'' guarantee of \cite{gilbert2005improved} cannot hold for a sub-linear, deterministic algorithm \cite{cohen2009compressed}. Moreover, the range of $k$ for which this algorithm is faster than the FFT is smaller in practice than that of \cite{gilbert2005improved}.

Most recently, the authors of \cite{hassanieh2012simple} presented a randomized algorithm that extends by an order of magnitude the range of sparsity for which it is faster than the FFT.  This is accomplished by removing the iterative aspect from \cite{gilbert2005improved} by using more efficient filters, which are nearly flat within the passband and which decay exponentially outside.  In contrast, the box-car filters used in \cite{gilbert2005improved} have a frequency response which oscillates and decays like $|\omega|^{-1}$. In addition, the identification of significant frequencies is done by direct estimation after hashing into a large number of bins rather than the binary search technique of \cite{gilbert2005improved}.  These changes give a runtime bound of $\bigo(\log N \sqrt{Nk\log N})$ and a somewhat stronger error bound $\|\hat{x}-\hat{y}\|_\infty^2 \le \eps k^{-1} \|\hat{x}-\hat{x}_\textup{opt}\|_2^2 + \delta \|\hat{x}\|_1^2$ with probability $1-1/N$, where $\eps>0$ and $\delta = N^{-\bigo(1)}$ is a precision parameter.

These existing algorithms generally take one of two approaches to the sparse Fourier transform problem.  In~\cite{gilbert2002near,akavia2003proving,gilbert2005improved,hassanieh2012simple}, the spectrum of the input is randomly permuted and then run through a low-pass filter to isolate and identify frequencies which carry a large fraction of the signal's energy.  This leads to randomized algorithms that fail on a non-negligible set of possible inputs.  On the other hand,~\cite{iwen2010combinatorial} takes advantage of the combinatorial properties of \emph{aliasing} in order to identify the significant frequencies.  This leads to a deterministic algorithm with higher runtime and sampling requirements than the randomized algorithms mentioned. Both of these randomized and deterministic approaches have drawbacks.  Randomized algorithms are not suitable for failure-intolerant applications, while the process used to reconstruct significant frequencies in~\cite{iwen2010combinatorial} relies on the Chinese Remainder Theorem (CRT), which is highly unstable to errors in the residues. While there do exist algorithms for ``noisy Chinese Remaindering'' \cite{goldreich2000chinese,boneh2002finding,shparlinski2004noisy} these have thus far not found application to the sparse DFT problem, and we leave this as future work.

As this paper was being prepared, the authors became aware of an independent work using very similar methods for frequency estimation in the noiseless case \cite{hassanieh2012nearly}. Both methods consider the phase difference between Fourier samples to extract frequency information, but are based on different techniques for binning significant frequencies. The authors of \cite{hassanieh2012nearly} use random dilations and efficient filters of \cite{hassanieh2012simple}, whereas we use different sample lengths in the spirit of \cite{iwen2010combinatorial}. We believe both contributions are of interest, and reinforce the notion that exploiting phase information is critical for developing fast, robust algorithms for the sparse Fourier transform problem.

\subsection{New Results}
In this paper we describe a simple, deterministic algorithm that avoids reconstruction with the CRT.  We are thus able to avoid two pitfalls associated with existing algorithms.  Our method relies on sampling the signal in the time domain at slightly shifted points, and thus it assumes access to an underlying continuous-time signal.  The shifted time samples allow us to determine the value of significant frequencies in sub-sampled FFTs and also indicate when two or more frequencies have been aliased in such a sub-sampled FFT.  These two key facts allow us to significantly reduce (by up to two orders of magnitude) the average-case sampling and runtime complexity of the sparse FFT over a certain class of random signals.  Our worst-case bounds improve by a constant factor those of prior deterministic algorithms. We present both adaptive and non-adaptive versions of our algorithms.  If the application allows samples to be acquired adaptively (that is, dependent on previous samples), we are able to improve further on our average-case bounds.

The remainder of this paper is organized as follows. In section \ref{sec:prelim} we introduce notation and prove the technical lemmas underlying our algorithms. In section \ref{sec:algs} we introduce randomized and deterministic versions of our algorithm. In section \ref{sec:avgcase} we prove that our algorithm has average-case runtime and sampling complexities of $\Theta(k\log(k))$ and $\Theta(k)$, respectively. In section \ref{sec:empirical} we present the results of an empirical evaluation of our algorithm and compare its runtime and sampling requirements to competing algorithms. Finally in section \ref{sec:conclusion} we provide some concluding remarks and discuss ongoing work to appear in the future.

\section{Mathematical Background}
\label{sec:prelim}
\subsection{Preliminaries}

Throughout this work we shall be concerned with frequency-sparse band-limited signals $S:[0,1)\to\C$ of the form
\begin{equation}
\label{eq:signal}
S(t) = \sum_{j=1}^k a_j \e^{2\pi\i\omega_j t},
\end{equation}
where $\omega_j \in [-N/2,N/2) \cap \Z, \; a_j \in \C,$ and $k \ll N$. The  Fourier series of $S$ is given by
\begin{equation} \label{eq:fourierseries}
   \widehat{S}(\omega) = \int_0^1 S(t) \e^{-2\pi\i\omega t} \d t, \; \omega \in \Z,
\end{equation}
so that for signals of the form~\eqref{eq:signal} we have $\widehat{S}(\omega_j) = a_j$
and $\wh S(\omega) = 0$ for all other $\omega \in [-N/2,N/2) \cap \Z$. Given any finite  sequence $\bS=(s_0, s_1, \dots, s_{p-1})$ of length $p$ we define its Discrete Fourier transform (DFT) by
\begin{equation} \label{eq:dft}
    \widehat{\bm{S}}[h] ~=~
        \sum_{j=0}^{p-1} s_j \e^{\frac{2\pi\i jh}{p}}
        ~=~ \sum_{j=0}^{p-1} \bS[j] W_p^{jh},
\end{equation}
where $h = 0, 1, \ldots, p-1$, $\bS[j]:=s_j$ and
$W_p:=\e^{-\frac{2\pi\i}{p}}$ is the primitive $p$-th root of unity. The Fast Fourier
Transform (FFT) allows the computation of $\wh\bS$ in $\bigo(p\log p)$ steps.

We apply the DFT to discrete samples of $S(t)$ to compute the Fourier coefficients $a_j$
of $S(t)$. For an integer $p$ and real $\eps > 0$ we form discrete arrays of samples
of $S$ of length $p$ via
$$
    \bm{S}_p[j] = S\Bigl(\frac{j}{p}\Bigr), ~~
    \bm{S}_{p,\eps}[j] = S\Bigl(\frac{j}{p}+ \eps\Bigr),~~ j = 0, 1, \ldots, p-1.
$$
Now assume that all $\omega_j \wmod{p},~ 1\leq j \leq k$ are distinct. It is a simple
derivation to obtain
$$
      \wh\bS_{p}[h] = \left\{\begin{array}{cl}
                                pa_j & ~~h \equiv\omega_j \wmod p \\
                                0 & \mbox{~~otherwise}.
      \end{array}\right.
$$
By examining the peaks of $\wh\bS_p[h]$ we will be able to determine
$\{\omega_j \wmod{p}: 1\leq j \leq k\}$. Previous approaches applied the
Chinese Remainder Theorem to reconstruct $\{\omega_j\}$ by taking a suitable number
of $p$'s, which must overcome the problem of registrations to match up each
$\omega_j$ whenever a new $p$ is used (see e.g.
\cite{iwen2010combinatorial,iwen2011improved}). Our algorithm takes
a different approach using the shifted sub-samples. Note that
$$
      \wh\bS_{p,\eps}[h] = \left\{\begin{array}{cl}
                                pa_je^{2\pi \i\eps \omega_j} & ~~h \equiv\omega_j \wmod p \\
                                0 & \mbox{~~otherwise}.
      \end{array}\right.
$$
It follows that in this setting, for $h\equiv\omega_j \wmod p$ we have
$\frac{\wh\bS_{p,\eps}[h]}{\wh\bS_{p}[h]} = \e^{2\pi \i \eps\omega_j}$.
Hence
\begin{equation} \label{eq:omega1}
   2\pi\eps\omega_j \equiv \Arg\left(\frac{\widehat{\bm{S}}_{p,\eps}[h]}{\widehat{\bm{S}}_{p}[h]}\right)
   \wmod{2\pi},
\end{equation}
where $\Arg(z)$ denotes the phase angle of the complex number $z$ in $[-\pi, \pi)$.
Assume that we take $|\eps| \leq \frac{1}{N}$. Then $\omega_j$ is completely determined by
(\ref{eq:omega1}) as there will be no wrap-around aliasing, and
\begin{equation} \label{eq:omega}
   \omega_j = \frac{1}{2\pi\eps} \Arg\left(\frac{\widehat{\bm{S}}_{p,\eps}[h]}{\widehat{\bm{S}}_{p}[h]}\right).
\end{equation}
In fact, more generally, if we have an estimate of $\omega_j$,
say $|\omega_j| < \frac{L}{2}$, then
by taking $|\eps| \leq \frac{1}{L}$ the same reconstruction formula (\ref{eq:omega})
holds. The observation that by taking slightly shifted samples will allow us to identify
frequencies in $S(t)$ underlies the algorithms which follow, and the bulk of this paper analyzes various aspects of the proposed algorithms, such as efficiency and robustness.

One of the problems is that when $p<N$ it is possible that two or more distinct frequencies will have the same remainder modulo $p$.  In this case we say the frequencies are \emph{aliased} or \emph{collide} $\wmod p$.  In general, for $h \in \{0,\ldots,p-1\}$ and
the given signal $S(t)$ let
$I(S,h;p) :=\{j:~\omega_j \equiv h \wmod p\}$. Then we have
\begin{equation} \label{eq:aliasing}
\wh{\bm{S}}_p[h] = \sum_{\omega \equiv h \scriptsize{\wmod p}} \wh{S}(\omega)
               = p\sum_{j\in I(S,h;p)} a_j.
\end{equation}
When aliasing occurs reconstruction via (\ref{eq:omega}) is no longer valid. The aliasing phenomenon presents a serious challenge for any method with sub-linear sampling complexity.  In the next section we develop a simple test to determine whether or not aliasing has occurred in a $p$-length DFT, which then allows us to effectively overcome this challenge and develop provably correct sub-linear algorithms.

\subsection{Technical Lemmas}

To effectively apply the sub-sample idea in a Fourier algorithm one must first
overcome the aliasing challenge. Using shifted sub-samples gives us a simple yet extremely
effective criterion to determine whether or not aliasing has occurred at a given location in a $p$-length DFT without resorting to complicated combinatorial techniques.
Observe that complementing (\ref{eq:aliasing}) we have
\begin{equation} \label{eq:shiftaliasing}
\wh{\bm{S}}_{p,\eps}[h]
               = p\sum_{j\in I(S,h;p)} a_j \e^{2\pi\i\eps\omega_j}.
\end{equation}
It follows that
\begin{align} \label{eq:diffaliasing}
   \left|\wh{\bm{S}}_{p,\eps}[h]\right|^2 -\left|\wh{\bm{S}}_{p}[h]\right|^2
       = p^2&\sum_{j,l\in I(S,h;p)} a_j\overline{a_l} \e^{2\pi\i\eps(\omega_j-\omega_l)} \\
       &-p^2\Bigl|\sum_{j\in I(S,h;p)} a_j\Bigr|^2. \notag
\end{align}

\begin{lemma} \label{lem:aliasing}
    Let $p>1$ and $h\in\{0, 1, \dots, p-1\}$. Assume that $q=|I(S,h;p)|>1$, i.e.
$\omega_j \equiv h \wmod{p}$ for more than one $j$ in $S(t)$. Then we have the following:
\begin{itemize}
\item[\rm (A)]~ Let $\eps>0$ and $E:=\{\omega_j-\omega_m: j,m\in I(S,h;p)\}$. 
Suppose that all elements of
$\eps E$ are distinct $\wmod 1$. Then
$\big|\wh{\bm{S}}_{p,m\eps}[h]\big|  \neq \big|\wh{\bm{S}}_{p}[h]\big|$
for some $1 \leq m \leq q^2-q$.
\item[\rm (B)]~ For almost all $\eps>0$ we have
$\big|\wh{\bm{S}}_{p,\eps}[h]\big|  \neq \big|\wh{\bm{S}}_{p}[h]\big|$.
\end{itemize}
\end{lemma}
\begin{proof}
   The proof of part (B) is immediate from (\ref{eq:diffaliasing}). Observe that
$f(\eps):=\left|\wh{\bm{S}}_{p,\eps}[h]\right|^2 -\left|\wh{\bm{S}}_{p}[h]\right|^2$
is trigonometric polynomial in $\eps$, and it is not identically $0$ given that
$q=|I(S,h;p)|>1$. Thus it has at most finitely many zeros for $\eps\in [0,1)$, and hence
(B) is clearly true.

   We resort to the Vandermonde matrix to prove part (A). For simplicity we write
$f(t) = \sum_{\alpha\in E} c_\alpha \e^{2\pi \i \alpha t}$. 
Set $r_\alpha := \e^{2\pi \i \alpha\eps}$
where $\eps$ satisfies the hypothesis of the lemma, which implies that all
$r_j$ are distinct. Assume the claim of part (A) is false. Then we have
$f(m\eps)=0$ for all $0 \leq m \leq q^2-q$. Here $f(0)=0$ is automatic because
$\bS_{p,0} = \bS_p$. Thus we have
\begin{equation}  \label{vandermonde}
     \sum_{\alpha\in E} c_\alpha r_\alpha^m = 0, ~~~m=0, 1, \dots, q^2-q.
\end{equation}
But the cardinality of $E$ is at most $q^2-q+1$, which means that there are
at most $q^2-q+1$ terms in the sum in (\ref{vandermonde}). Because all $r_\alpha$
are distinct the matrix
$[r_\alpha^m]$ is a nonsingular Vandermonde matrix, and for (\ref{vandermonde})
to hold all $c_\alpha$ must be zero. This is clearly not the case, and a contradiction.
\end{proof}

\medskip
\noindent
{\bf Remark.}~ Any irrational $\eps$ or $\eps= \frac{a}{b}$ with $a, b$ coprime and
$b\geq 2N$ will satisfy the hypothesis of part (A) of Lemma \ref{lem:aliasing}. It is
also easy to show that in the special
case where all coefficients $a_j$ are real and $|I(S,h;p)|=2$, we have
$\big|\wh{\bm{S}}_{p,\eps}[h]\big|  \neq \big|\wh{\bm{S}}_{p}[h]\big|$ for any
$\eps = \frac{a}{b}$ with $a, b$ coprime and $b\geq N$.

\medskip

Lemma~\ref{lem:aliasing} allow us to determine whether aliasing has occurred by
whether $\big|\wh{\bm{S}}_{p,\eps}[h]\big|/\big|\wh{\bm{S}}_{p}[h]\big|=1$ for
a few values of $\eps$. It offers both a deterministic (part (B)) and a random (part (A)) procedure to identify aliasing in the sub-sampled DFTs. In practice we need to set a tolerance $\tau$ in  order to accept or reject frequencies according to the criterion
\begin{equation}
\label{eq:aliasing_test}
\left| \frac{\left| \wh{\bS}_{p,\eps}[h]\right|}{\left| \wh{\bS}_p[h]\right|} -1\right| \le \tau.
\end{equation}
We typically choose $\eps = 1/cN$ for some small constant $c\geq 2$, which would satisfy the hypothesis of part (A) of Lemma~\ref{lem:aliasing}. A tolerance on the order of $p/N$ works well in general, which is what we use in our experiments in section~\ref{sec:empirical} below.

In our algorithms we will take a number of sub-sampled DFTs of an input signal $S(t)$ of the form~\eqref{eq:signal}, whose lengths we denote $p_\ell$.  Lemma~\ref{lem:aliasing} allows us to determine whether or not two or more frequencies are aliased, so that we only add the non-aliased term to our representation. Since it is unlikely that two or more frequencies are aliased modulo two different sampling rates, using a different $p_\ell$ in a subsequent iteration lets us quickly discover all frequencies present in $S(t)$. Lemma~\ref{lem:numps} gives a worst-case bound on the number of $p_\ell$'s required by our deterministic algorithm to identify all $k$ frequencies in a given Fourier-sparse signal.  It is similar to~\cite[Lemma 1]{iwen2010combinatorial}, but with a smaller constant.  In its proof we use the CRT, which we quote here for completeness (see, e.g.,~\cite{niven1991introduction}).

\begin{thm}[Chinese Remainder Theorem]
    Any integer $n$ is uniquely specified modulo $N$ by its remainders modulo $m$ pairwise relatively prime numbers $p_\ell$, provided $\prod_{\ell=1}^m p_\ell \ge N$.
\end{thm}

\begin{lemma}
\label{lem:numps}
    Let $M>1$. It suffices to take $1+(k-1)\lfloor\log_{M}N\rfloor$ pairwise
    relatively prime $p_\ell$'s with $p_\ell
    \ge M$ to ensure that each frequency $\omega_j$ is isolated (i.e. not aliased) 
    $\wmod{p_\ell}$ for at least one $\ell$.
\end{lemma}
\begin{proof}
    Assume otherwise, namely that given $p_\ell$ for $\ell =1, 2, \dots, L$ with 
    $L > k\lfloor\log_{M}N\rfloor$ there exists some $\omega_j$ such that
    $\omega_j$ is aliased $\wmod{p_\ell}$. By the Pigeon Hole Principle there
    exists at least one $\omega_m \neq \omega_j$ such that
    $\omega_j-\omega_m \equiv 0 \wmod{p_\ell}$ at least $q$ times, where
    $q> \lfloor\log_{M}N\rfloor$. Without loss of generality we assume that
    $\omega_j-\omega_m \equiv 0 \wmod{p_\ell}$ for $\ell =1, 2, \dots, q$.
    Now by the fact that $p_\ell \geq M$ we have
    $$
    \prod_{\ell=1}^q p_\ell \ge M^q \ge N.
    $$
    By the CRT we would then have $\omega_j \equiv \omega_m \wmod N$, a contradiction.
\end{proof}

We remark that the algorithm in~\cite{iwen2010combinatorial} requires taking $1+2k\log_kN$ co-prime sample lengths, since that algorithm requires each $\omega$ to be isolated in at least half of the DFTs of length $p_\ell$.  This requirement stems from the fact that that algorithm cannot distinguish between aliased and non-aliased frequencies in a given sub-sampled DFT. Our worst-case bound is approximately a factor of two better, though in practice our algorithms never use all those sample lengths on random input.  The fact that we can tell which frequencies are ``good'' for a given $p_\ell$ allows us to construct our Fourier representation one term at a time, and quit when we have achieved a prescribed stopping criterion.

\section{Algorithms}
\label{sec:algs}

Both of our algorithms proceed along a similar course; in fact they differ only in the choice of the sample lengths $p_\ell$.  We assume that we are given access to the continuous-time signal $S(t)$ whose Fourier coefficients we would like to determine, and further that we can sample from $S$ at arbitrary points $t$ in unit time.  This is an appropriate model for analog signals, but not for discrete ones.  In the discrete case, one could interpolate between given samples to approximate the required $S$-values, though we have not implemented or analyzed this case.  (The same assumptions hold for the algorithms in~\cite{iwen2010combinatorial}, while those in~\cite{gilbert2002near, gilbert2005improved, hassanieh2012simple} are formulated purely in the discrete realm.) In this paper mainly limit ourselves to the noiseless case.  Though this is a highly unrealistic assumption, it permits a simple description of the underlying algorithm. In section \ref{sec:alg-noise} we discuss some of the problems associated with noisy signals and give a minor modification of our algorithm for low-level noise. A second manuscript in preparation addresses the issue of noise specifically, with more significant modifications to the algorithms described below.

\subsection{Non-adaptive}
Our algorithms start by choosing a sample length $p_1$ such that $p_1 \ge ck$ for some constant $c>1$.  For a fixed $\eps \le 1/N$, we then compute $\bm{\wh{S}}_p$ and $\bm{\wh{S}}_{p,\eps}$, sort the results by magnitude, and compute frequencies $\omega$ via~\eqref{eq:omega} for the $k$ largest coefficients in absolute value. We then check whether or not each of those frequencies is aliased via~\eqref{eq:aliasing}, and if it is not, we add it to our list.  The coefficient is given by the unshifted sample value $\wh{\bS}_p[h]$ at that frequency.  After this, we combine terms with the same frequency and prune small coefficients from the list. We then iterate until a stopping criterion is reached.  In the empirical study described in section~\ref{sec:empirical}, we stopped when the number of distinct frequencies in our list equalled the desired sparsity.

Our deterministic algorithm chooses $p_\ell$ to be the $\ell^\textup{th}$ prime greater than $ck$.  This ensures that all samples lengths are co-prime, at the expense of taking slightly more samples than necessary.  By Lemma~\ref{lem:numps}, $1+(k-1)\lfloor\log_{ck}N\rfloor$ such $p_\ell$s suffice to isolate every $\omega$ at least once.  This gives us worst-case sampling and runtime complexity on the same order as~\cite{iwen2010combinatorial}, though the results in section~\ref{sec:empirical} indicate that on average we significantly outperform those pessimistic bounds.

Our Las Vegas algorithm chooses $p_\ell$ uniformly at random from the interval $[c_1k, c_2k]$ for constants $1<c_1<c_2$.  In this case we cannot make a worst-case guarantee on the number of iterations needed by the algorithm to converge.  However, the results in section~\ref{sec:empirical} indicate that the Las Vegas version performs similarly to the deterministic version on the class of signals tested.

\subsection{Adaptive}
The algorithms can also be implemented in an adaptive fashion, by which we mean that the size of the current representation is taken into account in subsequent iterations.  In particular, if $\bm{R}$ is our current representation, we let $k^* = k-|\bm{R}|$ and choose the next $p_\ell$ with respect to $k^*$ instead of $k$.  Moreover, before taking DFTs, we subtract off the contribution from the current representation, so that effort is not expended re-identifying portions of the spectrum already discovered.  This idea is similar to that in~\cite{gilbert2002near, gilbert2005improved}, though in our empirical studies the evaluation of the representation is done directly, rather than as an unequally-spaced FFT.  This gives our algorithms asymptotically slower runtime, but the effect is negligible for the values of $k$ studied in section~\ref{sec:empirical}.  A formal description appears below in algorithm~\ref{alg:alg}.

\begin{algorithm}
   \caption{\textsc{Phaseshift}}
   \label{alg:alg}
   \begin{algorithmic}[5]
       \STATE \textbf{Input:} function pointer $S$, integers $c_1,c_2,k,N$, real $\eps$
       \STATE \textbf{Output:} $\bm{R}$, a sparse representation for $\wh{S}$
       \STATE $\bm{R} \gets \emptyset$, $\eps_0 \gets 0, \eps_1 \gets \eps,$ $\ell \gets 1$
       \WHILE{$\left|\bm{R}\right| < k$}
           \STATE $k^* \gets k-|\bm{R}|$ \hfill \COMMENT{or $k$ if non-adaptive}
           \STATE $p_\ell \gets$ first prime $\ge c_1k^*$ \\ \hfill \COMMENT{or \textsc{Uniform}$(c_1k^*, c_2k^*)$ if Las Vegas}
           \FOR{$m=0$ to 1}
           		\FOR{$j=0$ to $\ell-1$}
           		\STATE $\bm{S}_{\ell,m}[j] \gets S\left(\dfrac{j}{p_\ell}+\eps_m\right)$
           		\STATE $\displaystyle{\bm{S}_\textrm{rep}[j] \gets \sum_{(\omega,c_\omega) \in \bm{R}} c_\omega \e^{2\pi\i\omega(j/p_\ell+\eps_m)}}$ \\ \hfill \COMMENT{omit if non-adaptive}
				\ENDFOR
           \STATE $\wh{\bm{S}}_{\ell,m} \gets \textsc{FFT}(\bm{S}_{\ell,m}-\bm{S}_\textrm{rep})$

           \STATE $\wh{\bm{S}}_{\ell,m}^\textrm{sort} \gets \textsc{Sort}(\wh{\bm{S}}_{\ell,m})$

           \FOR{$j=1$ to $k^*$}
               \STATE $\omega_{j,\ell} \gets \dfrac{1}{2\pi\eps} \Arg\left( \dfrac{\wh{\bm{S}}_{\ell,1}^\textrm{sort}[j]}{\wh{\bm{S}}_{\ell,0}^\textrm{sort}[j]}\right)$
           \ENDFOR
       \ENDFOR
       \FOR{$j=1$ to $k^*$}
           \IF{ $\left|\dfrac{\left|\wh{\bm{S}}_{\ell,0}^\textrm{sort}[j]\right|}{\left|\wh{\bm{S}}_{\ell,1}^\textrm{sort}[j]\right|}-1\right| < \dfrac{p_\ell}{N}$}
               \STATE $\bm{R} \gets \bm{R} \cup \left\{ \left( \omega_{j,\ell}, \wh{\bm{S}}_{\ell,0}[\omega_{j,\ell}] \right)\right\}$
           \ENDIF
       \ENDFOR
       \STATE collect terms in $\bm{R}$ with same $\omega$
       \STATE prune small coefficients from $\bm{R}$
       \STATE  $\ell \gets \ell+1$
       \ENDWHILE
   \end{algorithmic}
\end{algorithm}

\subsection{Modifications in the presence of noise}
\label{sec:alg-noise}
In the noiseless versions of the algorithms described in this paper, a test for aliasing is implemented by considering the ratio of magnitudes of shifted and unshifted peaks.  When the samples are corrupted by noise, there will be two challenges. 
The first challenge is that the reconstruction of frequencies from shifts
will be corrupted by noise. The second challenge is that there will be variations among the magnitudes even  for non-aliased terms, so a higher threshold that depends on the size of the noise must be set. When this threshold is too large it affects the ability to distinguish
aliased terms as there will be an increased number of false negatives. On the other hand, 
lower thresholds that reduce false negatives will lead to an increased number of false
positives. 

The first challenge can be addressed effectively through a combination of
using larger $p_j$'s, multiple shifts and a multiscale unwrapping. The
idea of using larger $p_j$'s is rather straightforward yet effective.
For any given $p_j$ the DFT detects the location of the frequencies
modulo $p_j$ rather accurately even with substantial noise. Furthermore, the reconstructed frequencies will still tend to cluster around the true value. Suppose that
we sample the signal and compute DFTs of length $p_j$ on these samples.  The locations of the peaks in these short DFTs tell us the accurate value of $\omega \bmod p_j$ for 
each unaliased frequency $\omega$ appearing in the signal.  Writing $\omega = ap_j + b$ with $a,b\in\Z$, we now know $b$ and must determine $a$.  

With a small amount of noise the reconstructed frequencies  $\wt{\omega}$ using \eqref{eq:omega} will be close to the true $\omega$. We can thus round $\wt{\omega}$ to the nearest integer of the form $ap_j + b$, which will recover the true frequency $\omega$ as long as $|\wt{\omega}-\omega| < p_j/2$. For high noise levels, it is possible that the $\wt{\omega}$ will deviate by more than $p_j/2$ from $\omega$, so that the value for $a$ given by rounding will be incorrect. By choosing larger $p_j$ (i.e., increasing the parameter $c_1$) one can alleviate the problem somewhat, provided that the noise level is not too high. When the noise level is so high that taking a large $p_j$ is no longer economical, a potential solution is to take multiple shifts and employ a multiscale unwrapping technique. We are still at the preliminary stage in our study of these new techniques, but early results are very encouraging.

The second challenge poses a bigger problem, but again it can be addressed in several
ways. The multiscale unwrapping method will repeatedly check for aliasing at each stage, which makes it highly unlikely that an aliased frequency will pass through all the tests. Even in the unlikely even that it does, our algorithm allows false
positives. Since each mode is subtracted from the original signal in our algorithm,
a false positive frequency will lead to an extra mode in the new signal. As the process
continues it will be extracted and cancel out the false frequency extracted earlier.

\section{Average-case analysis}
\label{sec:avgcase}
In this section we prove that the average-case runtime and sampling complexity of our algorithm are $\Theta(k\log k)$ and $\Theta(k)$, respectively. This is shown over a class of random signals described in section \ref{sec:random_signal_model}. Before giving this result on the expected runtime and sampling complexity, in section \ref{sec:inner_loop} estimate the costs of a single iteration of the while loop in algorithm \ref{alg:alg}, lines 5--25. We then describe in section \ref{sec:random_signal_model} the random signal model over which we prove our average-case bounds. In section \ref{sec:markov} we prove that the expected number of iterations of the while loop is constant, and in section \ref{sec:karp} we use this result to prove our average-case bounds.

\subsection{While loop runtime and sampling complexity}
\label{sec:inner_loop}
The computational cost of the while loop in algorithm \ref{alg:alg}, lines 5--25 is dominated by three operations. The first is the evaluation of the current representation $\bm{R}$ of $k-k*$ terms at the $\bigo(k^*)$ points $j/p_\ell$ in line 10. In our implementation, we simply calculated this directly, looping over both the sample points and the terms in the representation.  The complexity of this implementation is $\bigo(p_\ell (k-k^*)) = \bigo(k^*(k-k^*)) = \bigo(k^2)$, and while non-equispaced fast Fourier transforms \cite{dutt1993fast, anderson1996rapid} yield an asymptotically faster runtime of $\bigo(k \log(k))$, they also incur large overhead costs. For the values of $k$ considered in this paper, the direct evaluation seems to have little effect on the overall runtime. The other two dominant computational tasks in the inner loop are the FFTs of $\bigo(k)$ samples and the subsequent sorting of these DFT coefficients. It is well-known that both of these operations can be done in time $\Theta(k\log(k))$ \cite{cormen2001introduction}. Thus the inner loop has overall time complexity $\Theta(k \log(k))$, assuming the use of non-equispaced FFTs.

\subsection{Random signal model}
\label{sec:random_signal_model}
For both the average-case analysis and for the empirical evaluation described in section \ref{sec:empirical} we considered test signals with uniformly random phase over the bandwidth and coefficients chosen uniformly from the complex unit circle. In other words, given $k$ and $N$, we choose $k$ frequencies $\omega_j$ uniformly at random (without replacement) from $[-N/2,N/2)\cap\Z$. The corresponding Fourier coefficients $a_j$ are of the form $\e^{2\pi\i\theta_j}$, where $\theta_j$ is drawn uniformly from $[0,1)$.  The signal is then given by
\begin{equation}
    \label{eq:test_signal}
    S(t) = \sum_{j=1}^k a_j \e^{2\pi\i\omega_j t}.
\end{equation}
This is the standard signal model considered in previous empirical evaluations of sub-linear Fourier algorithms \cite{iwen2007empirical, iwen2010combinatorial, hassanieh2012simple}.

\subsection{Markov Analysis of Collisions}
\label{sec:markov}
In order to analyze the expected runtime and sampling complexity of our algorithms, we must estimate the expected number of collisions among frequencies modulo the sample lengths used by the algorithms. Recall that in the noiseless case, our algorithms are able to detect when a collision between two or more frequencies has occurred, and for those that are not aliased we are able to calculate the value of the frequency. Thus we seek to estimate the expected fraction of frequencies that are aliased modulo a given sample length $p$, since this determines how many passes the algorithm makes. In this section we derive bounds on the expected value of this quantity and discuss how the stopping criteria used in the algorithm affect its average-case performance.

In the random signal model considered in section \ref{sec:empirical}, we assume the $k$ frequencies are uniformly distributed over the bandwidth $[-N/2,N/2)$, and so the residues $\omega \bmod p$ are also uniformly distributed over $[0,p-1]$. Our problem then becomes a classical occupancy problem: the number of collisions among the frequencies is equivalent to the number of multiple-occupancy bins when $k$ balls are thrown uniformly at random into $p$ bins. Define $X_m$ to be the number of single-occupancy bins after $m$ balls are thrown, $Y_m$ to be the number of multiple-occupancy bins after $m$ balls are thrown, and $Z_m$ to be the number of zero-occupancy bins after $m$ balls are thrown.  Since $p$ is constant, we have the trivial relationship $Z_m = p-X_m-Y_m$, so it suffices to consider only the pair $(X_m,Y_m)$. When the $(m+1)^\textup{st}$ ball is thrown, we have the following possibilities:
\begin{itemize}
	\item it lands in an unoccupied bucket, with probability $Z_m/p = 1-(X_m+Y_m)/p$;
	\item it lands in a single-occupancy bucket, with probability $X_m/p$;
	\item it lands in a multiple-occupancy bucket, with probability $Y_m/p$.
\end{itemize}
In the first case, we have $X_{m+1}=X_m+1, \; Y_{m+1} = Y_m$; in the second case, we have $X_{m+1} = X_m-1, \; Y_{m+1}=Y_m+1$; and in the third case, we have $X_{m+1}=X_m, \; Y_{m+1}=Y_m$. Conditioning on the values of $X_m,\,Y_m$ we have
\begin{equation}
\label{eq:EXY_cond}
\E\left( \left[\begin{array}{c}X_{m+1} \\Y_{m+1}\end{array}\right]\left| \left[\begin{array}{c}X_{m} \\Y_{m}\end{array}\right]\right.\right) 
= \left[\begin{array}{cc}1-2/p & -1/p \\1/p & 1\end{array}\right]\left[\begin{array}{c}X_m \\Y_m
\end{array}\right] + \left[\begin{array}{c}1 \\0\end{array}\right],
\end{equation}
so that the system forms a Markov chain. By recursively conditioning on the values of $X_{m-1},\,Y_{m-1}$, we can calculate the expected values of $X_k,\,Y_k$ for any $k>0$ using the initial condition $X_1=1,\,Y_1=0$. Denoting by $A$ the matrix in the right-hand side of equation~\eqref{eq:EXY_cond}, we have
\begin{equation}
\label{eq:EXYk_A}
\E\left( \left[\begin{array}{c}X_k \\Y_k\end{array}\right]\right) = \sum_{m=0}^{k-1}\left(A^m \left[\begin{array}{c}1 \\0\end{array}\right] \right)
= \left( \sum_{m=0}^{k-1}A^m\right) \left[\begin{array}{c}1 \\0\end{array}\right].
\end{equation}
Since $\rho(A) = 1-1/p < 1$, where $\rho$ is the spectral radius, the geometric matrix series can be written
\begin{equation}
\label{eq:geomat}
\sum_{m=0}^{k-1}A^m = (I-A)^{-1}(I-A^k).
\end{equation}
After some linear algebra, we obtain
\begin{equation}
\label{eq:EXYk_final}
\E\left( \left[\begin{array}{c}X_k \\Y_k\end{array}\right]\right) 
= \left[\begin{array}{c}k(1-\frac{1}{p})^{k-1} \\p(1-(1-\frac{1}{p})^k)-k(1-\frac{1}{p})^{k-1}\end{array}\right].
\end{equation}
Since $Z_k=p-X_k-Y_k$, we have $\E\left(Z_k\right) = p(1-1/p)^k$.

In our algorithms we choose $p=ck$ for some small integer $c$. Using this and the approximation $(1+\frac{x}{n})^n \approx \e^x$, we have
\begin{equation}
\label{eq:EXYk_approx}
\E\left( \left[\begin{array}{c}X_k \\Y_k\end{array}\right]\right) \approx \left[\begin{array}{c}k\e^{-1/c} \\ck(1-\e^{-1/c})-k\e^{-1/c}\end{array}\right].
\end{equation}
This gives a nonlinear equation for the expected number of collisions among $k$ frequencies as a function of the parameter $c$. Newton's method can then be used to determine the value $c$ required to ensure a desired fraction of the frequencies are not aliased. For example, to ensure that 90\% of frequencies are isolated on average, it suffices to take $c=5$; this value for the parameter $c$ had already been found to give good performance in our empirical evaluation of the algorithms. 


\subsection{Average-case runtime and sampling complexity}
\label{sec:karp}
In this section we will use a probabilistic recurrence relation due to Karp \cite{karp1994probabilistic,dubhashi2009concentration} to give average-case performance bounds and concentration results for the case when the algorithm is halted after identifying $k$ or more terms. In particular, we use the following theorem for recurrences of the form
\begin{equation}
\label{eq:prob_recur}
T(k) = a(k) + T(H(k)),
\end{equation}
where $T(k)$ denotes the time required to solve an instance of size $k$, $a(k)$ is the amount of work done on a problem of size $k$, and $0\le H(k) \le k$ is a random variable denoting the size of the subproblem generated by the algorithm.

\begin{thm}{\cite[Theorem 1.2]{karp1994probabilistic}}
\label{thm:karp}
Suppose $a(k)$ is nondecreasing, continuous, and strictly increasing on $\{x:a(x)>0\}$, and that $\E[H(k)] \le m(k)$ for a nondecreasing continuous function $m(k)$ such that $m(k)/k$ is also nondecreasing. Denote by $u(k)$ the solution to the deterministic recurrence
\begin{equation}
\label{eq:det_recur}
u(k) = a(k) + u(m(k)). 
\end{equation}
Then for $k>0$ and $t\in\N$,
\begin{equation}
\label{eq:karp_thm}
\P[T(k) > u(k) + ta(k)] \le \left( \frac{m(k)}{k}\right)^t.
\end{equation}
\end{thm}

Our algorithm does work $a(k) = \Theta(k\log(k))$ on input of size $k$ and generates a subproblem whose average size is $m(k) = k/10$. (Recall from section \ref{sec:markov} that with the parameter $c=5$, on average over 90\% of the frequencies were not aliased modulo $p = \bigo(ck)$.) The associated deterministic recurrence is then
\begin{equation}
\label{eq:our_det_recur}
u(k) = \Theta(k\log(k)) + u(k/10),
\end{equation}
whose solution is $u(k) = \Theta(k\log(k))$ (see, e.g., \cite{cormen2001introduction}). A straightforward application of Theorem \ref{thm:karp} yields 
\begin{equation}
\label{eq:conc_bound}
\P[T(k) > \Theta(k\log(k)) + t k\log(k)] \le 10^{-t},
\end{equation}
so that the runtime is tightly concentrated about its mean $\Theta(k\log(k))$. 

The sampling complexity $S(k)$ can be handled in an analogous manner, since in this case $a(k) = \Theta(k)$ and $m(k) = k/10$ as before. The associated deterministic recurrence becomes 
\begin{equation}
\label{eq:our_det_samp_recur}
u(k) = \Theta(k) + u(k/10),
\end{equation}
whose solution is $u(k) = \Theta(k)$. Applying Theorem \ref{thm:karp} again we have
\begin{equation}
\P[S(k) > \Theta(k) + tk] \le 10^{-t},
\end{equation}
so that we again have tight concentration of the number of samples around the mean $\Theta(k)$.

%

\section{Empirical Evaluation}
\label{sec:empirical}
In this section we describe the results of an empirical evaluation of the \emph{adaptive} deterministic and Las Vegas variants of the Phaseshift algorithm described above.  Both algorithms were implemented in C++ using FFTW 3.0 \cite{frigo2005design} for the FFTs, using \texttt{FFTW\_ESTIMATE} plans since the sample lengths are not known in advance for the Las Vegas variant.  For comparison we also ran the same tests on the four variants of GFFT as well as on AAFFT and FFTW itself.  The FFTW runs utilized the \texttt{FFTW\_PATIENT} plans with wisdom enabled, and so are highly optimized. The experiments were run on a single core of an Intel Xeon E5620 CPU with a clock speed of 2.4 GHz and 24 GB of RAM, running SUSE Linux with kernel 2.6.16.60-0.81.2-smp for x86\_64. All code was compiled with the Intel compiler using the \texttt{-fast} optimization. As in~\cite{iwen2011improved}, timing is reported in CPU ticks using the \texttt{cycle.h} file included with the source code for FFTW.

In the following sections we refer to our algorithm as ``Phaseshift'', since by taking shifted time samples of the input signal we also shift the phase of the Fourier coefficients.  To keep the plots readable, we only show data for the adaptive, deterministic variant of our algorithm; the other variants perform similarly  The algorithms of~\cite{iwen2011improved} are denoted GFFT-XY, where X $\in \{$D,R$\}$ and Y $\in\{$F,S$\}$.  The D/R stands for deterministic or randomized, while the F/S stands for fast or slow.  The fast variants use more samples but less runtime while the slow variants use fewer samples but more runtime.  In the plots below, we always show the GFFT variant with the most favorable sampling or runtime complexity. Finally, AAFFT denotes the algorithm of~\cite{gilbert2005improved}. The implementations tested are summarized in table~\ref{tab:implementations} along with the average-case sampling and runtime complexities, and the associated references.
\begin{table}
	\centering
	\caption{Implementations used in the empirical evaluation.}
		\begin{tabular}{c|c|c|c|c}
		\hline
		Algorithm & R/D & Samples  & Runtime  & Reference \\
		\hline
		PS-Det & D & $k$ & $k\log k$ & Section \ref{sec:avgcase} \\
		PS-LV  & R & $k$ & $k\log k$ & Section \ref{sec:avgcase} \\
		GFFT-DF& D & $k^2\log^4 N$ & $k^2\log^4 N$ & \cite{iwen2011improved} \\
		GFFT-DS& D & $k^2 \log^2 N$ & $N k\log^2 N$ & \cite{iwen2011improved} \\
		GFFT-RF& R & $k\log^4 N$ & $k\log^5 N$ & \cite{iwen2011improved} \\
		GFFT-RS& R & $k\log^2 N$ & $N \log N$ & \cite{iwen2011improved} \\
		AAFFT  & R & $k \log^c N$ & $k \log^c N$ & \cite{gilbert2005improved} \\
		FFTW & D & $N$ & $N\log N$ & \cite{frigo2005design} \\
		\end{tabular}
	\label{tab:implementations}
\end{table}

\subsection{Setup}
Each data point in Fig.~\ref{fig:fixed_n}--\ref{fig:fixed_k} is the average of 100 independent trials of the associated algorithm for the given values of the bandwidth $N$ and the sparsity $k$. The lower and upper bars associated with each data point represent the minimum and maximum number of samples or runtime of the algorithm over the 100 test functions. The values of $k$ tested were $2, 4, 8, \ldots, 4096$, while the values of $N$ were $2^{17}, 2^{18}, \ldots, 2^{26}$.  For larger values of $k$, the slow GFFT variants and AAFFT took too long to complete on our hardware, so we only present partial data for these algorithms. Nevertheless, the trend seen in the plots below continues for higher values of the sparsity. The test signals were generated according to the signal model described in section \ref{sec:random_signal_model}.

The Phaseshift and deterministic GFFT variants will always recover such signals exactly.  The randomized GFFT variants are Monte Carlo algorithms, and so, when they succeed, will also recover the signal exactly.  AAFFT, on the other hand, is an approximation algorithm which will fail on a non-negligible set of input signals.  However, for the runs depicted in Fig.~\ref{fig:fixed_n}--\ref{fig:fixed_k}, AAFFT always produced an answer with $\ell_2$ error less than $10^{-4}$.  The randomized GFFT variants failed a total of 7 times out of 2200 test signals, a relatively small amount that can be reduced by parameter tuning. For the Phaseshift variants, we chose the parameters $c_1 = 5, \; c_2 = 10$, and took the shift $\eps$ to be $1/2N$.  Finally, for the randomized GFFT variants, we chose the Monte Carlo parameter to be $1.2$.

\subsection{Sampling Complexity}
\label{sec:sampling}
In Fig.~\ref{fig:fixed_n} (a), we compare the average number of samples of the input signal $S$ required by each algorithm when the bandwidth $N$ fixed at $2^{22}$.  The sparsity of the test signal is varied from 2 to 4096 by powers of two.  We can see that the Phaseshift variants require over an order of magnitude fewer samples than GFFT-RS, the GFFT variant with the lowest sampling requirements.  Both Phaseshift variants also require over an order of magnitude fewer samples than AAFFT.  The comparison with the deterministic GFFT variants is even starker; Phaseshift-Det requires two orders of magnitude fewer samples than GFFT-DS (not shown), and four orders of magnitude fewer samples than GFFT-DF (not shown).

In Fig.~\ref{fig:fixed_k} (a), we compare the average number of samples of the input signal $S$ required by each algorithm when the sparsity $k$ is fixed at 60.  The bandwidth $N$ was varied from $2^{17}$--$2^{26}$ by powers of two.  Using powers of two for the bandwidth allows the best performance for both FFTW and AAFFT, though this fact is more relevant for the runtime comparisons in the following section.  We can see that the Phaseshift variants require many fewer samples than all four GFFT variants as well as AAFFT and FFTW, for all values of $N$ tested.  The Phaseshift variants exhibit almost no dependence on the bandwidth for all values of $N$, a feature not shared by the other deterministic algorithms. We note here that in future work we plan to replace the $1/2N$ shift by two or more larger shifts with co-prime denominators to obtain an equivalent shift, as in~\cite{wang1998use}.  This should lead to more robustness at high values of $N$.

\subsection{Runtime Complexity}
\label{sec:runtime}
In Fig.~\ref{fig:fixed_n} (b), we compare the average runtime of each algorithm over 100 test signals when the bandwidth $N$ is fixed at $2^{22}$.  The range of sparsity $k$ considered is the same as in section~\ref{sec:sampling}. For all values of $k$ the Phaseshift variants are faster than GFFT-RF (the fastest GFFT variant) and AAFFT by more than an order of magnitude. When compared to GFFT-RS (not shown), GFFT-DS (not shown), and FFTW, the difference in runtime is closer to three orders of magnitude.

In Fig.~\ref{fig:fixed_k} (b), we compare the average runtime of each algorithm over 100 test signals when the sparsity $k$ is fixed at 60.  The range of bandwidth considered is the same as in section~\ref{sec:sampling}. The Phaseshift variants are the only algorithms that outperform FFTW for all values of $N$ tested. The other implementations tested only become competitive with the standard FFT for $N \gtrsim 2^{20}$, while ours are faster even for modest $N$.

\begin{figure}
	\centering
		\subfloat[]{\includegraphics[width=0.8\textwidth]{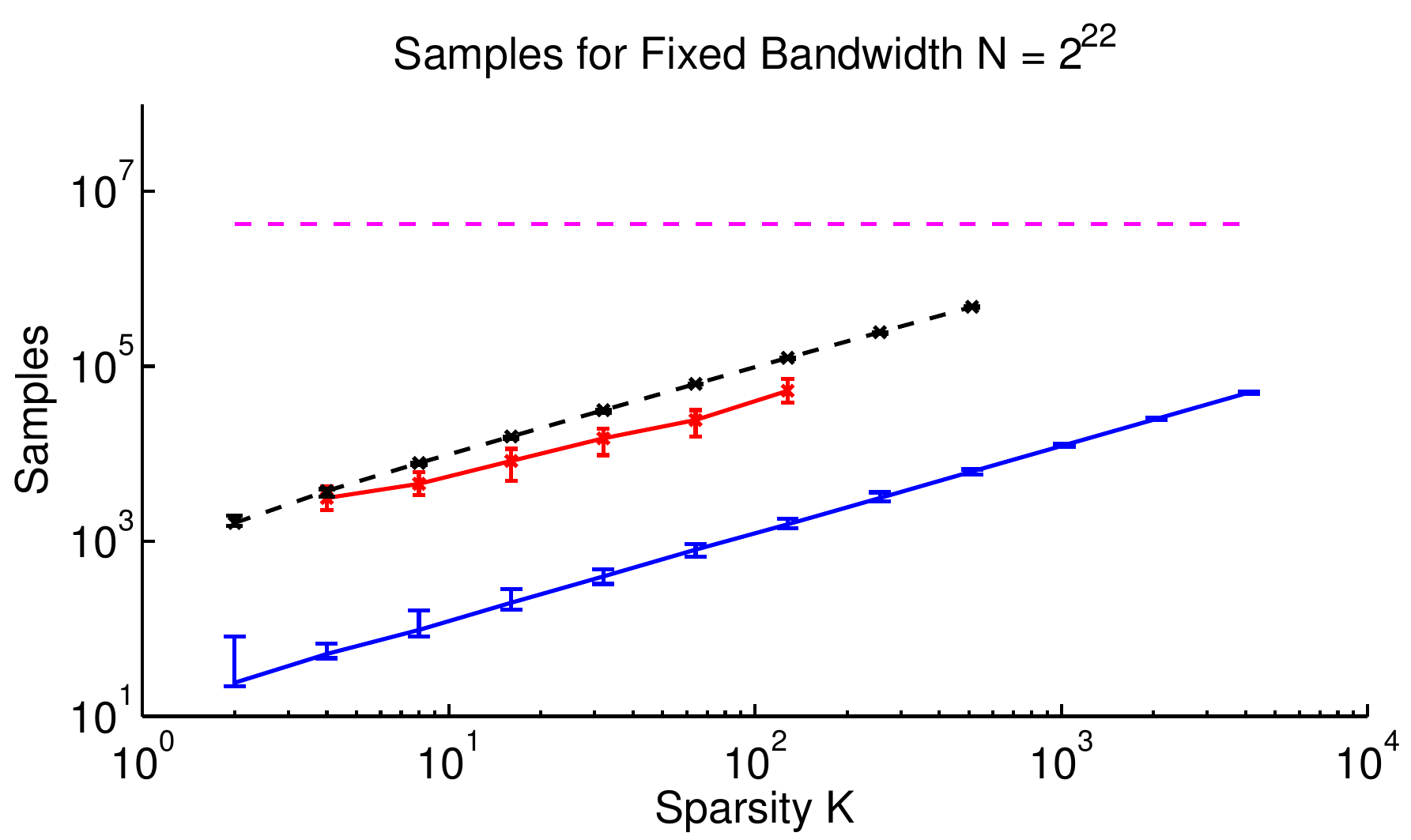}}\\
		\subfloat[]{\includegraphics[width=0.8\textwidth]{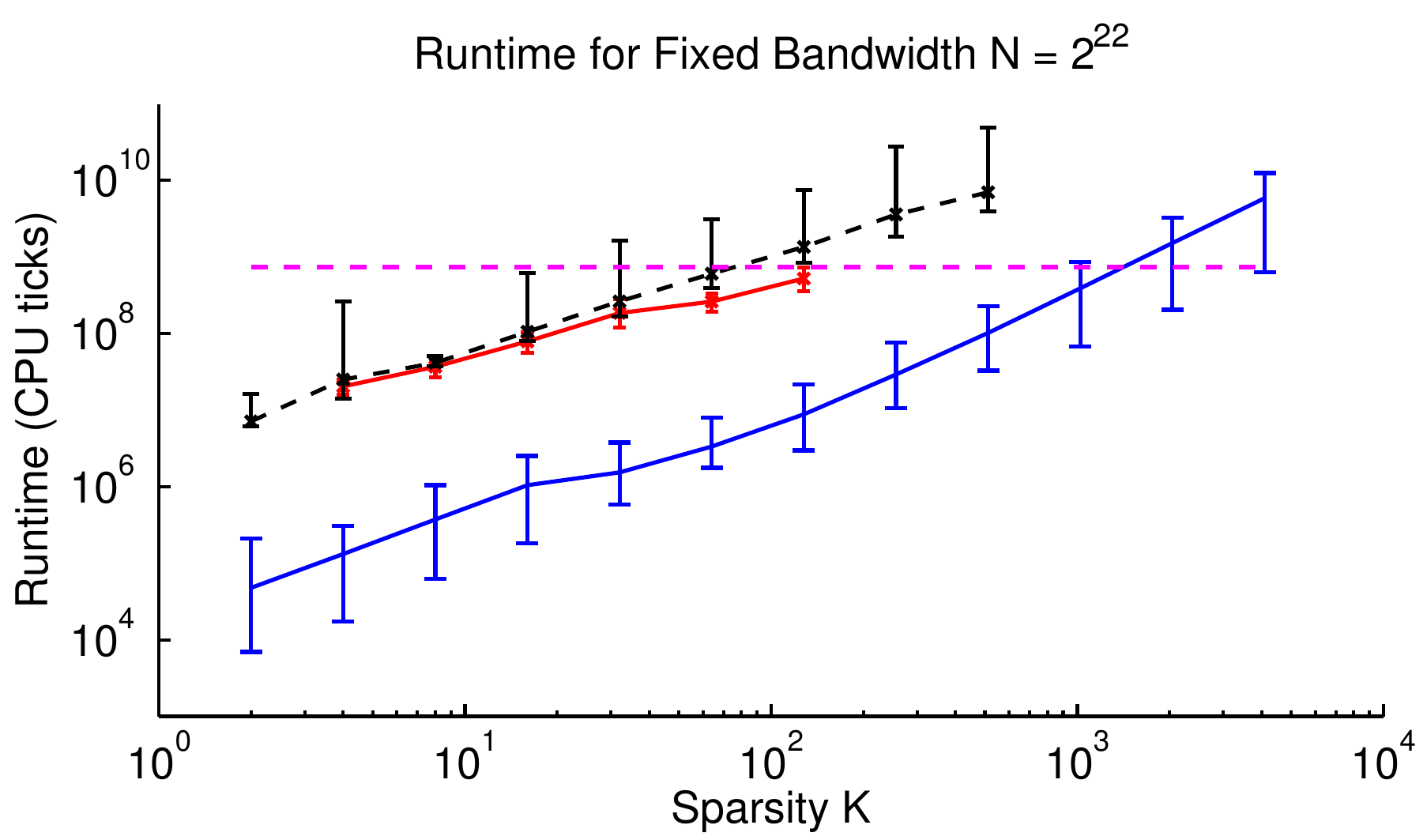}}
		\caption{(a) Sampling complexity with fixed bandwidth $N=2^{22}$ for PS-Det (blue solid line), GFFT-RS (red solid line), AAFFT (black dashed line), and FFTW (magenta dashed line). (b) Runtime complexity with fixed bandwidth $N=2^{22}$ for PS-Det (blue solid line), GFFT-RF (red solid line), AAFFT (black dashed line), and FFTW (magenta dashed line).}
		\label{fig:fixed_n}		
\end{figure}

\begin{figure}
	\centering
		\subfloat[]{\includegraphics[width=0.8\textwidth]{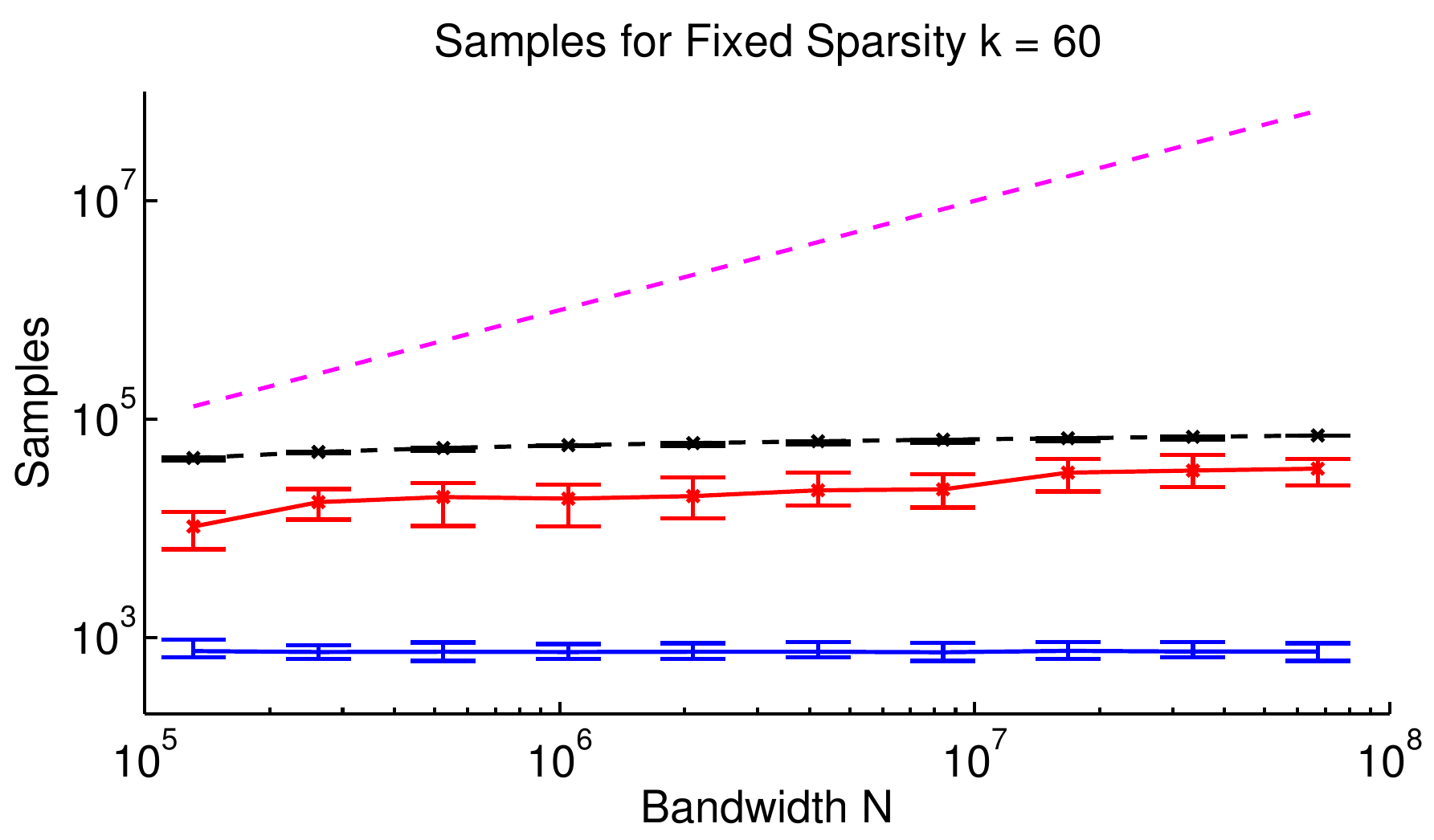}}\\
		\subfloat[]{\includegraphics[width=0.8\textwidth]{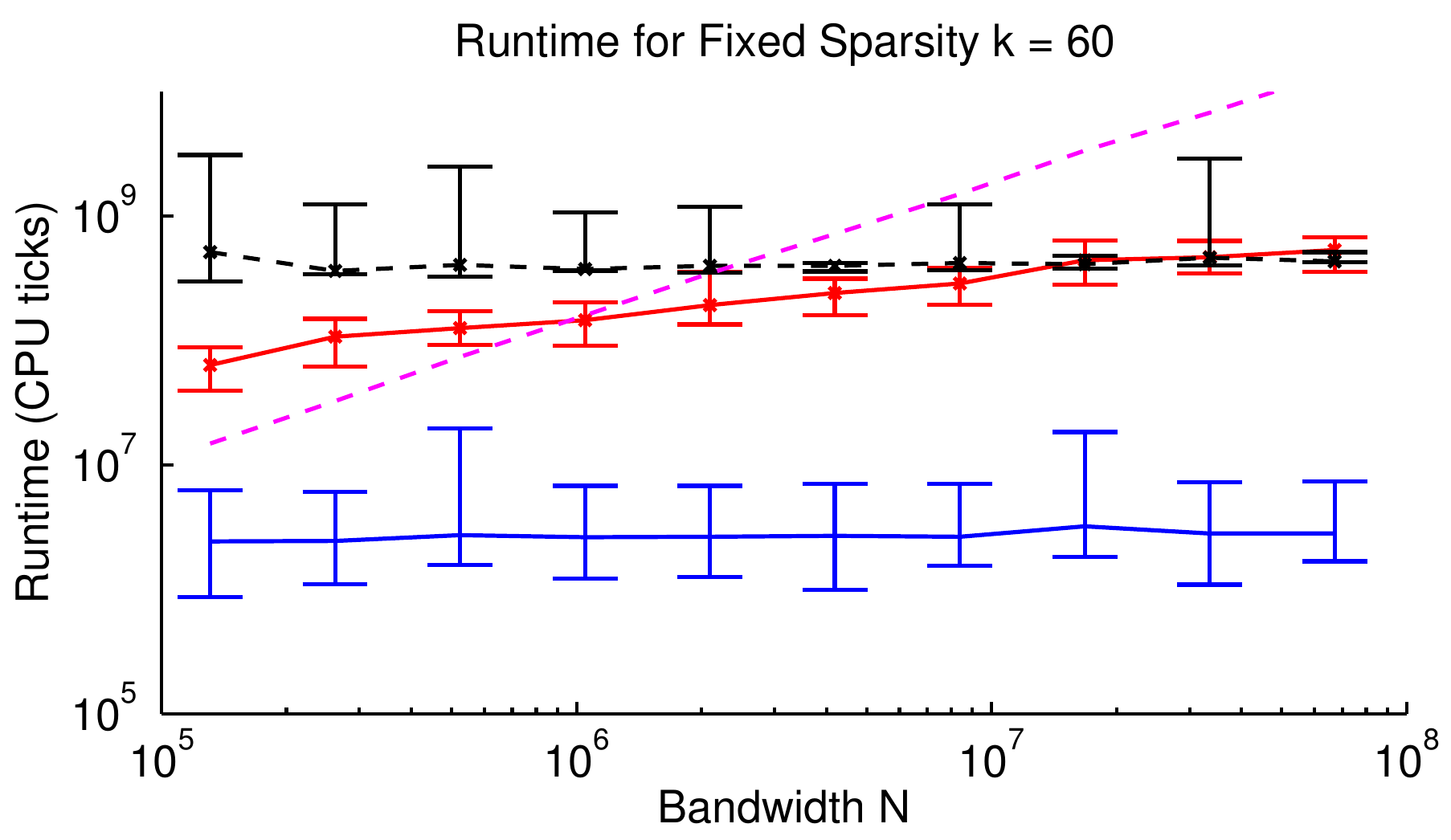}}
		\caption{(a) Sampling complexity with fixed sparsity $k=60$ for PS-Det (blue solid line), GFFT-RS (red solid line), AAFFT (black dashed line), and FFTW (magenta dashed line). (b) Runtime complexity with fixed sparsity $k=60$ for PS-Det (blue solid line), GFFT-RF (red solid line), AAFFT (black dashed line), and FFTW (magenta dashed line).}
		\label{fig:fixed_k}
\end{figure}

\subsection{Noisy Case}
We report here on a preliminary study of the performance of the deterministic algorithm in the presence of noise. Our noisy signals were of the same form as in the previous section, but with  complex white gaussian noise of standard deviation $\sigma$ added to each measurement. As described in section \ref{sec:alg-noise}, the simplest way to deal with low-level noise is to simply round the reconstructed frequencies to the nearest integer of the form $ap_j + b$, where $b \equiv \omega \bmod p_j$ is the location of the peak in a length-$p_j$ DFT. This modification doesn't change the runtime or sampling complexity significantly, so in this section we focus on the error in the approximation as a function of the noise level $\sigma$ and the parameter $c_1$.

In the existing literature on the sparse Fourier transform, the $\ell_2$ norm is most often used to assess the quality of approximation. There are many reasons for this choice, with the the two most convincing perhaps being the completeness of the complex exponentials with respect to the $\ell_2$ norm and Parseval's theorem. For certain applications, however, this choice of norm is inappropriate. For example, in wide-band spectral estimation and radar applications, one is interested in identifying a set of frequency intervals containing active Fourier modes. In this case, an estimate $\wt{\omega}$ of the true frequency $\omega$ with $|\wt{\omega}-\omega| \ll N$ is useful, but unless $\wt{\omega}=\omega$ the $\ell_2$ metric will report an $\bigo(1)$ error. Furthermore, when considering non-periodic signals (equivalently, non-integer $\omega$'s) the same precision problem appears when using the $\ell_2$ metric.

For these reasons, we propose measuring the approximation error of sparse Fourier transform problems with the Earth Mover Distance (EMD) \cite{rubner2000earth}. Originally developed in the context of content-based image retrieval, EMD measures the minimum cost that must be paid (with a user-specified cost function) to transform one distribution of points into another. EMD can be calculated efficiently as the solution of a linear program corresponding to a certain flow minimization problem. In our situation, we consider the cost to move a set of estimated Fourier modes and coefficients $\left\{(\wt{\omega}_j,c_{\wt{\omega}_j})\right\}_{j=1}^{\wt{k}}$ to the true values $\left\{(\omega_i, c_{\omega_j})\right\}_{j=1}^k$ under the cost function
\begin{equation}
\label{eq:emd-cost}
d_1\big( (\omega,c_\omega), (\wt{\omega},c_{\wt{\omega}}); N \big) \eqdef \frac{|\omega-\wt{\omega}|}{N} + |c_\omega-c_{\wt{\omega}}|.
\end{equation}
This choice of cost function strikes a balance between the fidelity of the frequency estimate (as a fraction of the bandwidth) and that of the coefficient estimate. We denote the EMD using $d_1$ for the cost function by EMD(1) below.

In figure \ref{fig:error} we report the average EMD(1) error over 100 test signals as a function of the input noise level $\sigma$, for various choices of the parameter $c_1$. In this experiment, the sparsity and bandwidth are fixed at $k=64$ and $N=2^{22}$, respectively. As expected, the error decreases as $c_1$ increases, since the rounding procedure described in section \ref{sec:alg-noise} is more likely to result in the true frequency. Moreover, the error increases linearly with the noise level, indicating the procedure's robustness in the presence of noise. 

\begin{figure}
	\centering
	\includegraphics[width=0.8\textwidth]{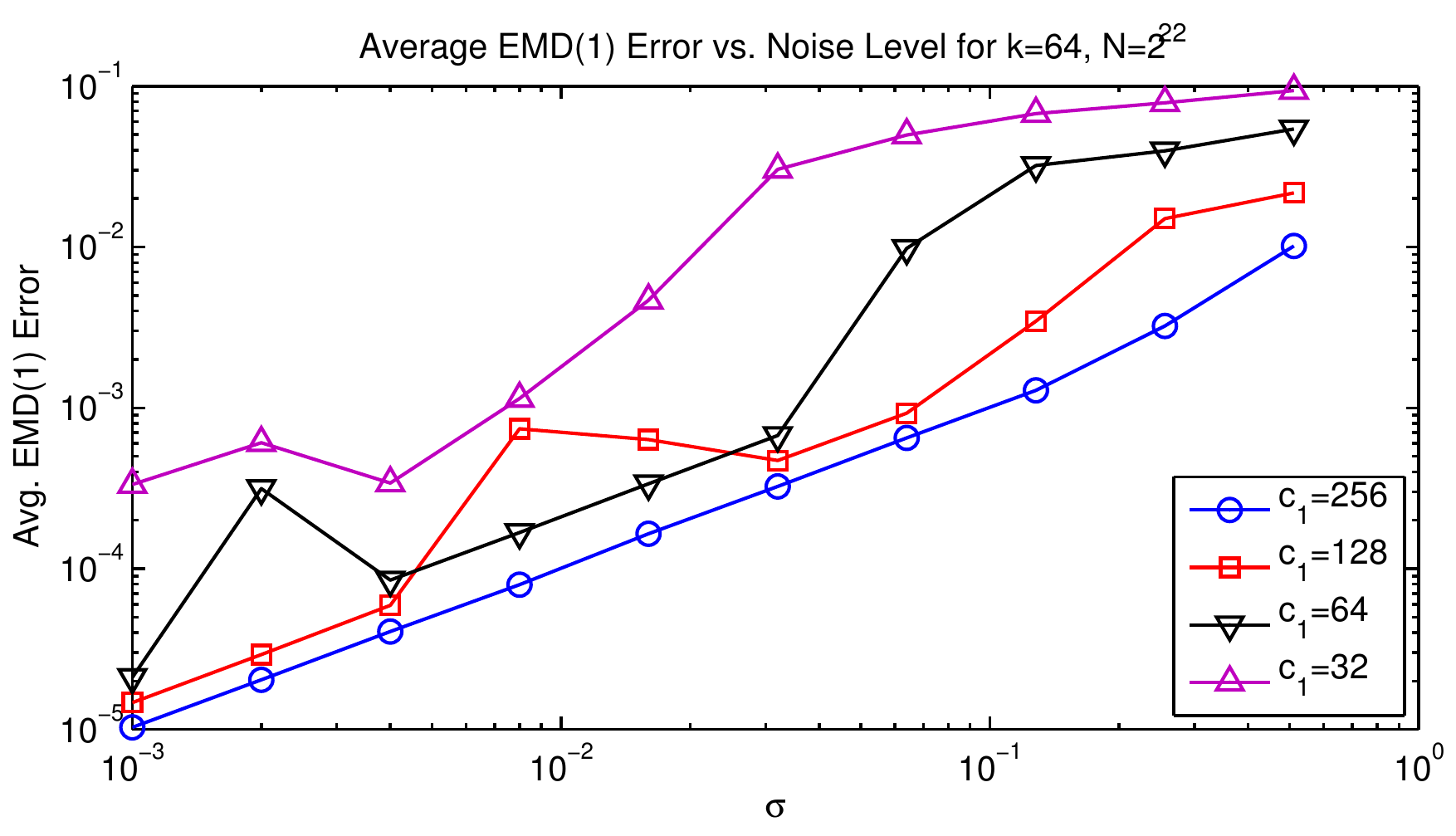}
	\caption{EMD(1) error as a function of the noise level $\sigma$ for various choices of the parameter $c_1$. The sparsity and bandwidth are fixed at $k=64$, $N=2^{22}$, respectively.}
	\label{fig:error}
\end{figure}

We remark that in the noiseless case the choice $c_1=5$ was found to be sufficient, while figure \ref{fig:error} indicates that the much larger value $c_1\approx 256$ is necessary for good approximation in the EMD(1) metric. The larger sample lengths imply an increase in both the runtime and sampling complexity, and indicate that the rounding procedure of section \ref{sec:alg-noise} should be complemented by other modifications. This is the purpose of a second manuscript under preparation, in which we combine the rounding procedure with the use of larger shifts $\eps_j$ in a multiscale approach to frequency estimation.

\section{Conclusion}
\label{sec:conclusion}

In this paper we have presented a deterministic and Las Vegas algorithm for the sparse Fourier transform problem that empirically outperform existing algorithms in average-case sampling and runtime complexity. While our worst-case bounds do not improve the asymptotic complexity, we are able to extend by an order of magnitude the range of sparsity for which our algorithm is faster than FFTW in the average case. The improved performance of our algorithm can be attributed to two major factors: adaptivity and ability to detect aliasing. In particular, we are able to extract more information from a small number of function samples by considering the \emph{phase} of the DFT coefficients in addition to their magnitudes. This represents a significant improvement over the current state of the art for the sparse Fourier transform problem.

We have developed a multiresolution approach to handle the noisy case, in which we learn the value of a frequency from most to least significant bit by increasing the size of the shift $\eps$. Finally, we are exploring the extension of these methods to handle non-integer frequencies, which would represent the first such result in the sparse Fourier transform context.

\section*{Acknowledgments}
We would like to thank Mark Iwen and I. Ben Segal for making available the source code to the AAFFT and GFFT algorithms, Yossi Rubner for making available the source code for the Earth Mover Distance, and Piotr Indyk and Eric Price for sharing a preprint and source code for the sFFT algorithm.  We also acknowledge helpful discussions with Anna Gilbert and Martin Strauss.

\bibliographystyle{amsalpha}
\bibliography{comprehensive}

\end{document}